\documentclass[reqno,12pt]{amsart}

\usepackage{enumerate}
\usepackage[margin=1in]{geometry}
\usepackage{ifpdf}
\usepackage{amsmath}
\usepackage{amsfonts}
\usepackage{amssymb}
\usepackage{amsthm}
\usepackage{amscd}
\usepackage[hyperfootnotes=false]{hyperref}
\usepackage{setspace}
\usepackage{amsrefs}
\usepackage{nicefrac}
\usepackage{makecell}
\usepackage{mathscinet}

\newcommand{\ignore}[1]{}

\newcommand{\sH}{{\mathcal{H}}}

\newtheorem{thm}{Theorem}[section]

\newtheorem{prop}[thm]{Proposition}

\theoremstyle{definition}

\theoremstyle{remark}

\author{Ji\v{r}\'{\i} Lebl}
\thanks{The first author was in part supported by NSF grant DMS 0900885.}
\address{Department of Mathematics, Oklahoma State University,
Stillwater, OK 74078, USA}
\email{lebl@math.okstate.edu}

\date{October 8, 2013}

\ifpdf
\hypersetup{
  pdftitle={Addendum to Uniqueness of certain polynomials constant on a line},
  pdfauthor={Jiri Lebl}
}
\fi

\title{Addendum to Uniqueness of certain polynomials constant on a line}

\begin{document}

\begin{abstract}
The computer calculations in \cite{LL} to classify sharp polynomials
with nonegative coefficients constant on the line $x+y=1$
have been extended to degrees 19 and 21.
In degree 19 a surprisingly large number of 13 sharp polynomials was found,
while in degree 21 only the group invariant polynomial exists.
\end{abstract}

\maketitle


\section{Introduction} \label{section:intro}

Let $\sH(2,d)$ be the space of polynomials $p(x,y)$ of two variables with
nonnegative coefficients such that $p(x,y)=1$ whenever $x+y=1$.
It is known~\cite{DKR} that the degree $d$ satisfies $d \leq 2N-3$, where $N$
is the number of nonzero coefficients of $p$.  Furthermore,
for each odd degree $d$, there exists a group invariant
polynomial with precisely $d=2N-3$.  We call polynomials satisfying equality
the \emph{sharp polynomials}.  See \cites{DKR,LL} for more information,
background, and motivation.

In \cite{LL} using computer code, we have have classified
the
sharp polynomials in $\sH(2,d)$ up to degree 17.
Due to the increase in the speed of computers over the last several years and
improvements to the computer code to be described below, it was
possible to run the computer code for degrees 19 and 21.  The
computation with all the improvements takes approximately 5 days 
for degree 19 on a
relatively recent 4-core CPU .  For degree 21, the computation
took over 8 months.

The computer has found 13 sharp polynomials polynomials in degree 19 up to
symmetry of the variables, several of which were unexpected.
Two of the polynomials are symmetric.  The number of
sharp polynomials in degree 19
is stunning as there have been only 16 sharp polynomials in all odd
degrees up to degree 17.

In degree 21, the computer code found that up to swapping of variables,
there are no sharp polynomials
besides the group invariant one, which 
was previously known.

We have therefore computed one new term in the sequence
A143106 on OEIS~\cite{OEIS:uniq}, of degrees where the group invariant
polynomial is the unique one up to swapping of variables.  The sequence is
now known to be:
\begin{equation}
1, 3, 5, 9, 17, 21 .
\end{equation}
We have also computed two new terms for the sequence A143107 on
OEIS~\cite{OEIS}.  That is, a sequence whose $N$th term is the number of sharp
polynomials of degree $2N-3$, or in other words, the number of polynomials
in $\sH(2,2N-3)$ with $N$ nonzero coefficients.  In this sequence symmetry
is not taken into account and therefore there are 24 polynomials altogether
for degree 19 and 2 polynomials for degree 21.  The sequence is now known to be:
\begin{equation}
0, 1, 1, 2, 4, 2, 4, 8, 4, 2, 24, 2 .
\end{equation}

\section{The polynomials}

Let us start with degree 21.  In this degree only the group invariant
polynomial exists.  That is, up to swapping of variables the only sharp
degree 21 polynomial is the group invariant one.
\begin{multline}
x^{21}
+ 21\,x^{19}\,y^1
+ 189\,x^{17}\,y^2
+ 952\,x^{15}\,y^3
+ 2940\,x^{13}\,y^4
+ 5733\,x^{11}\,y^5
+ 7007\,x^9\,y^6
+\\
5148\,x^7\,y^7
+ 2079\,x^5\,y^8
+ 385\,x^3\,y^9
+ 21\,x^1\,y^{10}
+ y^{21}
\end{multline}

In degree 19, the situation is dramatically different.  Let us list the
polynomials we found.  
First we have the group invariant polynomial.
\begin{multline}
{x}^{19}+19\,y\,{x}^{17}+152\,{y}^{2}\,{x}^{15}+665\,{y}^{3}\,{x}^{13}+1729\,{y}^{4}\,{x}^{11}+2717\,{y}^{5}\,{x}^{9}+2508\,{y}^{6}\,{x}^{7}+\\
1254\,{y}^{7}\,{x}^{5}+285\,{y}^{8}\,{x}^{3}+19\,{y}^{9}\,x+{y}^{19}
\end{multline}
Then we have several polynomials with integer coefficients.
\begin{multline}
{x}^{19}+19\,y\,{x}^{17}+152\,{y}^{2}\,{x}^{15}+665\,{y}^{3}\,{x}^{13}+1729\,{y}^{4}\,{x}^{11}+2090\,{y}^{5}\,{x}^{9}+627\,{y}^{9}\,{x}^{5}+\\
627\,{y}^{5}\,{x}^{5}+285\,{y}^{8}\,{x}^{3}+19\,{y}^{9}\,x+{y}^{19}
\end{multline}
\begin{multline}
{x}^{19}+285\,{y}^{3}\,{x}^{13}+1425\,{y}^{4}\,{x}^{11}+19\,{y}^{9}\,{x}^{9}+2679\,{y}^{5}\,{x}^{9}+19\,y\,{x}^{9}+2508\,{y}^{6}\,{x}^{7}+\\
1254\,{y}^{7}\,{x}^{5}+285\,{y}^{8}\,{x}^{3}+19\,{y}^{9}\,x+{y}^{19}
\end{multline}
\begin{multline}
{x}^{19}+285\,{y}^{3}\,{x}^{13}+1425\,{y}^{4}\,{x}^{11}+19\,{y}^{9}\,{x}^{9}+2052\,{y}^{5}\,{x}^{9}+19\,y\,{x}^{9}+627\,{y}^{9}\,{x}^{5}+\\
627\,{y}^{5}\,{x}^{5}+285\,{y}^{8}\,{x}^{3}+19\,{y}^{9}\,x+{y}^{19}
\end{multline}
Next, three polynomials with rational coefficients with denominator 25.
\begin{multline}
{x}^{19}+19\,y\,{x}^{17}+152\,{y}^{2}\,{x}^{15}+\frac{15371\,{y}^{3}\,{x}^{13}}{25}+\frac{6137\,{y}^{4}\,{x}^{11}}{5}+\frac{4807\,{y}^{5}\,{x}^{9}}{5}+\\
\frac{1254\,{y}^{13}\,{x}^{3}}{25}+\frac{4617\,{y}^{8}\,{x}^{3}}{25}+\frac{1254\,{y}^{3}\,{x}^{3}}{25}+19\,{y}^{9}\,x+{y}^{19}
\end{multline}
\begin{multline}
{x}^{19}+\frac{5871\,{y}^{3}\,{x}^{13}}{25}+\frac{4617\,{y}^{4}\,{x}^{11}}{5}+19\,{y}^{9}\,{x}^{9}+\frac{4617\,{y}^{5}\,{x}^{9}}{5}+19\,y\,{x}^{9}+\\
\frac{1254\,{y}^{13}\,{x}^{3}}{25}+\frac{4617\,{y}^{8}\,{x}^{3}}{25}+\frac{1254\,{y}^{3}\,{x}^{3}}{25}+19\,{y}^{9}\,x+{y}^{19}
\end{multline}
\begin{multline}
{x}^{19}+\frac{5871\,{y}^{3}\,{x}^{13}}{25}+\frac{4617\,{y}^{4}\,{x}^{10}}{5}+19\,{y}^{9}\,{x}^{9}+\frac{4617\,{y}^{6}\,{x}^{9}}{5}+19\,y\,{x}^{9}+\\
\frac{1254\,{y}^{13}\,{x}^{3}}{25}+\frac{4617\,{y}^{8}\,{x}^{3}}{25}+\frac{1254\,{y}^{3}\,{x}^{3}}{25}+19\,{y}^{9}\,x+{y}^{19}
\end{multline}
Then we have two polynomials with denominator 56.
\begin{multline}
{x}^{19}+\frac{855\,y\,{x}^{17}}{56}+\frac{646\,{y}^{2}\,{x}^{15}}{7}+\frac{1938\,{y}^{3}\,{x}^{13}}{7}+\frac{2907\,{y}^{4}\,{x}^{11}}{7}+\\
\frac{3553\,{y}^{5}\,{x}^{9}}{14}+\frac{323\,{y}^{8}\,{x}^{3}}{7}+\frac{209\,{y}^{17}\,x}{56}+\frac{323\,{y}^{9}\,x}{28}+\frac{209\,y\,x}{56}+{y}^{19}
\end{multline}
\begin{multline}
{x}^{19}+\frac{855\,y\,{x}^{17}}{56}+\frac{323\,{y}^{2}\,{x}^{15}}{7}+\frac{323\,{y}^{8}\,{x}^{9}}{7}+\frac{323\,{y}^{5}\,{x}^{9}}{2}+\\
\frac{323\,{y}^{2}\,{x}^{9}}{7}+\frac{323\,{y}^{8}\,{x}^{3}}{7}+\frac{209\,{y}^{17}\,x}{56}+\frac{323\,{y}^{9}\,x}{28}+\frac{209\,y\,x}{56}+{y}^{19}
\end{multline}

Finally, rather surprisingly, there are 4 very similar polynomials with
denominator 110, two of which are symmetric in $x$ and $y$.
\begin{multline}
{x}^{19}+\frac{19\,y\,{x}^{17}}{2}+\frac{323\,{y}^{2}\,{x}^{15}}{11}+\frac{323\,{y}^{3}\,{x}^{13}}{11}+\frac{323\,y\,{x}^{6}}{55}+\\
\frac{323\,{y}^{13}\,{x}^{4}}{11}+\frac{323\,{y}^{14}\,{x}^{2}}{11}+\frac{19\,{y}^{17}\,x}{2}+\frac{323\,{y}^{6}\,x}{55}+\frac{399\,y\,x}{110}+{y}^{19}
\end{multline}
\begin{multline}
{x}^{19}+\frac{19\,y\,{x}^{17}}{2}+\frac{323\,{y}^{2}\,{x}^{15}}{11}+\frac{323\,{y}^{3}\,{x}^{13}}{11}+\frac{323\,y\,{x}^{6}}{55}+\\
\frac{323\,{y}^{13}\,{x}^{3}}{11}+\frac{323\,{y}^{15}\,{x}^{2}}{11}+\frac{19\,{y}^{17}\,x}{2}+\frac{323\,{y}^{6}\,x}{55}+\frac{399\,y\,x}{110}+{y}^{19}
\end{multline}
\begin{multline}
{x}^{19}+\frac{19\,y\,{x}^{17}}{2}+\frac{323\,{y}^{2}\,{x}^{14}}{11}+\frac{323\,{y}^{4}\,{x}^{13}}{11}+\frac{323\,y\,{x}^{6}}{55}+\\
\frac{323\,{y}^{13}\,{x}^{4}}{11}+\frac{323\,{y}^{14}\,{x}^{2}}{11}+\frac{19\,{y}^{17}\,x}{2}+\frac{323\,{y}^{6}\,x}{55}+\frac{399\,y\,x}{110}+{y}^{19}
\end{multline}
\begin{multline}
{x}^{19}+\frac{19\,y\,{x}^{17}}{2}+\frac{323\,{y}^{2}\,{x}^{14}}{11}+\frac{323\,{y}^{4}\,{x}^{13}}{11}+\frac{323\,y\,{x}^{6}}{55}+ \\
\frac{323\,{y}^{13}\,{x}^{3}}{11}+\frac{323\,{y}^{15}\,{x}^{2}}{11}+\frac{19\,{y}^{17}\,x}{2}+\frac{323\,{y}^{6}\,x}{55}+\frac{399\,y\,x}{110}+{y}^{19}
\end{multline}

\section{No adjacent terms}

One new optimization used to compute degree 21 requires a proof.
For terminology see \cites{LL,LP,DKR}.
Previously the code has avoided polynomials with both terms
$x^{j+1}y^k$ and $x^{j}y^{k+1}$.  If $p(x,y)$ is sharp and
contains both terms, one can obtain via an 
undoing a sharp polynomial with one of the terms missing.
If we could undo both, then the polynomial could not have been sharp to
begin with.  However, using recent work, \cite{LP}, we can prove a stronger assertion,
to improve efficiency of the code.

\begin{prop}
If $p \in \sH(2,d)$, $d > 1$ and odd, is sharp, then $p$ contains no
adjacent terms.  That is, given any $j$ and $k$, at least two of the
monomials of the form $x^{j+1}y^k$, $x^{j}y^{k+1}$, and $x^jy^k$ do not
appear in $p$.
\end{prop}

\begin{proof}
In \cite{LL} it was proved that $p$ must contain terms $x^d$, $y^d$.
The Newton diagram (see \cite{LL}) of $q(x,y) = \frac{p(x,y)-1}{x+y-1}$
must therefore contain all $P$s on the sides leading to pure monomials.

The top row (degree $d-1$ in $q$) must contain all
terms.  In \cite{LL} we have proved that the top row must be
alternating $N$s and $P$s.  This fact follows from the observation that the
only degree $d$ terms in $p$ are $x^d$ and $y^d$.

For $d > 1$ then by the above we see that any adjacent terms would have
to occur somewhere in the interior.  We also assume that adjacent
terms occur at the same degree.  For example if we got
$x^{j+1}y^k$, and $x^jy^k$, we could multiply the 
$x^jy^k$ term by $(x+y)$ to obtain another sharp polynomial with
two adjacent terms of same degree.

As $p$ is sharp,
all terms in $p-1$ must correspond to
sinks and sources.
Using terminology from \cite{LP}, the Newton diagram must be connected
as $p-1$ has only one negative term.
Furthermore, it was proved
in \cite{LP} that the number of sources and sinks satisfies the
same bound for connected diagrams as those arising from $\sH(2,d)$.  In
other words,
there can be at most $\frac{d+5}{2}$ sinks and sources.  Therefore
if we fill in any zeros with $P$s or $N$s or flip signs,
without increasing the number of
sinks and sources, we cannot in fact decrease the number of sinks and
sources as it is already minimal.  So we can never be in a situation
where flipping $P$s and $N$s or setting zeros to $P$s and $N$s reduces
the number of sinks or sources.  Let us disqualify certain situations
by showing we could reduce the number of sinks or sources.

Suppose we have two sinks next to each other in a configuration such
as:
\begin{equation}
\begin{matrix}
\ & N &\ & N  \\
P &\ & P &\ & P \\
\ & P &\ & P 
\end{matrix}
\end{equation}
Here the rows correspond to a fixed degree of terms in $q$.  For example,
the middle row corresponds to the terms
$x^{j-1}y^{k+1}$, $x^{j}y^{k}$, and $x^{j+1}y^{k-1}$.

We could flip the middle $P$ (term corresponding to $x^j y^k$)
to an $N$ and decrease the number of sinks.
There of course could also be zeros present. However, any time the top
two sinks have both at least one $P$ or $N$ from a term that
does not correspond to the $x^jy^k$, we could set the term corresponding
to $x^j y^k$ to $N$ and remove two sinks.  If both sinks have just zeros
as in
\begin{equation}
\begin{matrix}
\ & 0 &\ & 0  \\
0 &\ & P &\ & 0 \\
\ & P &\ & P 
\end{matrix}
\end{equation}
simply switching the $P$ to a 0 will remove the two sinks.  Therefore,
we must have one of the sinks have simply zeros and we must have a $P$
at the $x^jy^k$ term (otherwise one of the sinks would not be there).  In
other words we have a situation such as
\begin{equation}
\begin{matrix}
\ & 0 &\ & N  \\
0 &\ & P &\ & P \\
\ & P &\ & P 
\end{matrix}
\end{equation}
Now we cannot just flip the $P$ to an $N$.  Doing so would kill one sink,
convert one to a source and create a new sink, so it would not lower the
number of sinks and sources.  So let us start filling.

Take the smallest degree $k$ for which $q$ has a term missing.  The row
corresponding to degree $k$ and $k-1$ will therefore have a place that has
a gap of zeros such as
\begin{equation}
\begin{matrix}
P &\ & 0 &\ & \cdots &\ & 0 &\ & P  \\
\ & P &\ & N & \cdots & N &\ & P 
\end{matrix}
\end{equation}
possibly with the $P$s and $N$s reversed.  We know we always have such
a situation, since the diagram is connected and all the sides
are already filled.

If we fill the row of 0s with
alternating $P$s and $N$s the total number of sinks and sources cannot increase.
We may have converted a sink to a source or vice versa, but we have not
increased the total number.

We keep filling until we get to the row that is the middle row in the
configuration
\begin{equation}
\begin{matrix}
\ & 0 &\ & N  \\
0 &\ & P &\ & P \\
\ & P &\ & P 
\end{matrix}
\end{equation}
we notice that we could start filling that row with an $N$ on the right hand
side and that would cancel one of the two adjacent sinks.  And this would
lead to a contradiction that the number of sinks plus sources was optimal.
\end{proof}

\section{The computer code}

The code used is the C code described in detail in \cite{LL}.
The new revision of the code
that was used in this computation has been posted at \cite{code}.
The basic idea is to consider $p(x,1-x)-1 = 0$, which provides a linear
equation for the coefficients.  We write this equation as a matrix $A$ that takes
coefficients of degree $d-1$ or less to the degree $d$ coefficients.  We
know that the degree $d$ coefficients have the form $x^d+y^d$.
We iterate over the list of
possible monomials of degree $d-1$ or less, taking the corresponding
columns of the matrix, we look for nonzero solutions.  The idea
of the algorithm is to find those submatrices that are not
of full rank (have a nontrivial solution).  Then we check if this
solution has positive coefficients.
There are several heuristics that are applied that can avoid
doing row reduction at all.
Already in \cite{LL},
to improve speed of the row reduction, we first used mod $p$
arithmetic for reduction for a small prime,
as most of the submatrices considered
are full rank.  This technique reduces the need to do row reduction in full
integer arithmetic in vast majority of the cases.  We have used $p=19$
in degree 17 or less, and
in this calculation, we used $p=23$.  The prime must not divide the degree, as
most entries in the matrix are divisible by $d$.

The major improvements done in this revision are the following
\begin{enumerate}[(i)]
\item The row reduction is first done mod 2 before being done mode $p$.  Mod 2
is much faster than mod $p$.  The code actually
does column reduction in mod 2
and considers columns as unsigned integers.  The 
reduction removes one internal loop as adding columns together
is simply an XOR operation.  Unfortunately mod 2 arithmetic only eliminates
90\% of the full rank cases, but it
is approximately 4 times as fast on these 90\%.  If
the test fails with mod 2, we move to mod $p$ as before.
\item Just as the Mathematica code from~\cite{LL}, the C code now
ignores polynomials that are ``right side heavy.''  That is, polynomials
where there are more terms with higher power of $y$ than of $x$.  This
optimization reduces the run time by approximately one third.
\item No adjacent terms can appear as was mentioned before.  Skipping
all these cases improved the runtime 
by another factor of one half.
\item Many other more minor optimizations were done, whose individual impact
was harder to measure.
\end{enumerate}

Overall, the new optimizations together with improvements in speed of
computers since 2008, the current code runs approximately 25--50 times faster
than it did in 2008.

\section{Candidate sequence for uniqueness}

In \cite{LL} we have stated that we tried the construction of section 8 up
to degree 513 and listed the degrees not ruled out up to degree 149.
The code was run to degree 1250 and it seems this is a proper place to
record the results.  Therefore, the construction of section 8 in \cite{LL}
and hence an extension of the list in Proposition 7.2 is:

1, 3, 5, 9, 17, 21, 33, 41, 45, 53, 69, 77, 81, 93, 105, 113, 117, 125, 129,
141, 149, 153, 161, 165, 177, 185, 201, 213, 221, 225, 249, 261, 269, 273,
285, 297, 305, 309, 333, 341, 345, 357, 365, 369, 381, 405, 413, 417, 429,
437, 441, 453, 465, 473, 489, 501, 521, 525, 537, 549, 581, 585, 597, 609,
617, 621, 633, 645, 653, 665, 689, 693, 701, 705, 725, 729, 741, 753, 765,
773, 777, 789, 809, 825, 833, 837, 845, 861, 881, 885, 897, 905, 909, 921,
933, 953, 957, 969, 981, 993, 1017, 1029, 1041, 1049, 1053, 1061, 1065,
1085, 1089, 1097, 1101, 1113, 1125, 1137, 1149, 1157, 1173, 1185, 1193,
1197, 1205, 1229, 1233


\def\MR#1{\relax\ifhmode\unskip\spacefactor3000 \space\fi%
  \href{http://www.ams.org/mathscinet-getitem?mr=#1}{MR#1}}

\end{document}